\catcode`\@=11
\documentclass[11pt]{amsart}

\usepackage{sasha}
\usepackage{sasha_prih}
\usepackage{rank1}
\usepackage[normalem]{ulem}

\input thmdefs.sty

\usepackage[usenames,dvipsnames]{color}
\usepackage{graphicx}

\title{On group actions with simple Lebesgue spectrum} 

\author[A.\ Prikhod'ko]{Alexander Prikhod'ko}
\address{Department of Mechanics and Mathematics, Moscow State University}
\email{sasha.prihodko@gmail.com}

\begin{document}
%

\begin{abstract}
In work \cite{LebesgueFlows} a set of ergodic flows with simple Lebesgue spectrum is found, 
and the construction of these flows is based on the phenomenon of existence of 
{\it Littlewood-type flat polynomials\/} with coefficients $0$ and~$1$ 
on the group~$\Set{R}$, which is closely related to the algebraic and arithmetic properties 
of $\Set{R}$ as {\it a~field}. Thus, if we think about a~hypothetical extension of this phenomenon 
to general Abelian groups and its futher applications to ergodic group 
actions with simple Lebesgue spectrum, the method used in~\cite{LebesgueFlows} could be applied to 
a very specific class of groups including, for example, $p$-adic fields. 
At~the~same time, in some cases it is posible to generalize the flatness phenomenon 
applying a~sort of straightforward analytic technique. 
In this paper we illustrate this argument and propose 
a method of such kind that helps to pass from the case of the group~$\Set{R}$ 
to its Cartesian product $\Set{R}^d$. 
We establish the existence of $\Set{R}^d$-actions with Lebesgue spectrum 
of multiplicity one using the synthesis of the original construction of a flow 
introduced in~\cite{LebesgueFlows} and this new analytic method.   

The work is supported by grants RFFI \No\,11-01-00759-a. 

\end{abstract}

\newcounter{nfigure}[section]
\renewcommand{\thenfigure}{\thesection.\arabic{nfigure}}

\maketitle

 

\maketitle 

\section[Introduction]{Introduction. Rank one flows with simple Lebesgue spectrum}

\subsection[Spectral invariants]{Spectral invariants of ergodic group actions}

Let us consider an invertible transformation $T$ of the standard Lebesgue probability space $(X,\mu)$ 
preserving measure~$\mu$. We require that ${T \Maps X \to X}$ is an invertible map 
such that both $T$ and $T^{-1}$ are measurable and ${\mu(TA) = \mu(T^{-1}A) = \mu(A)}$ 
for any measurable set~$A$. It follows from Rokhlin's theorem (see~\cite{Rokhlin49}) that 
without loss of generality we can assume that $(X,\mu)$ is 
the unit segment $[0,1]$ with the standard Lebesgue measure. 

The {\it Koopman operator\/} $\hat T$ in $L^2(X,\mu)$ associated with~$T$ (see~\cite{KSF,LemEncycloSpTh}) 
is defined as 
\begin{equation}
	\hat T \Maps \varphi(\x) \to \varphi(T\x), \qquad \x \in X, \quad \varphi \in L^2(X,\mu). 
\end{equation}
Clearly, $\hat T$ is a unitary operator, hence, by spectral theorem $\hat T$ is identified 
up to unitary equivalence by the measure of maximal spectral type~$\sigma$ 
and the multiplicity function $\Mult(z)$. The spectral type $\sigma$ of a unitary operator
is a Borel measure on the unit circle in the complex plane. A~great progress made 
in the spectral theory of measure preserving transformations and group actions during last years
(e.g.~see~\cite{Ageev05,AnosovOKratnostyah,DanilenkoRyzh,KatokLemOmMultFuncT,LemEncycloSpTh}), 
though the it is still a~complicated problem to classify the pairs $(\sigma,\Mult(z))$ 
that can appear as spectral invariants of some dynamical system. 

In this paper we deal mainly with measure preserving actions $T^t$ of the group $\Set{R}$, 
refered to as {\it measurable flows}, as well as $\Set{R}^d$-actions. 
The spectral invariants $(\sigma,\,\Mult(z))$ of the unitary representation 
${\hat T}^t$ associated with an $\Set{R}^d$-action are defined 
in the same way like in the case of a~single unitatry operator, but in this case the measure 
of maximal spectral type $\sigma$ is a Borel measure on ${{\Hat{\Set{R}}}^d = \Set{R}^d}$.

\subsection{Rank one flows}

Let us consider an increasing sequence $(h_n)_{n=1}^\infty$, ${h_n > 0}$, and 
the corresponding sequence of segments ${I_n = [0,h_{n+1}]}$. 
Suppose that for each~$n$ a~finite collection of {\em disjoint\/} subintervals 
\begin{equation}
	(\om_n(j),\om_n(j)+h_n) \subset I_{n+1}, \qquad j = 0,\dots,q_n-1, 
\end{equation}
is given such that ${\om_n(j) + h_n \le \om_n(j+1)}$, and define the corresponding projection 
${\phi_n \Maps I_{n+1} \to I_n}$ such that ${\phi_n(\om_n(j) + t) = t}$ for any $j$ and ${0 \le t \le h_n}$ 
and ${\phi_n(t_{n+1}) = 0}$ otherwise. Notice that $\phi_n$ is continuous as a map from 
$I_{n+1}$ to $I_n$ if we identify edge points $0$ and~$h_n$, i.e.\ 
if we consider $\phi_n$ as a continuious map of degree~$q_n$ from the circle $\Set{R} / h_{n+1}\Set{Z}$ 
to the circle $\Set{R} / h_n\Set{Z}$. The map $\phi_n$ is a~linear map with derivative~$1$ at 
any interval ${(\om_n(j),\om_n(j)+h_n)}$, and $\phi_n$ is constant on the complement to these intervals. 
In other words, $\phi_n$ is a formal representation of the following dynamical process: 
a~point $x_n(t)$ moves in the segment $I_n$ with the velocity~$1$ and 
after arriving to the right edge $h_n$ of the segment the point $x_n(t)$ waits for time 
\begin{equation}
	s_{n,j} = \om_n(j+1) - \om_n(j) - h_n
\end{equation}
depending on the index $j$, 
and after this time is passed continues moving from the left edge $0$ of the segment~$I_n$. 
The values $s_{n,j}$ in terms of cutting-and-stacking construction%
\footnote{
	For the common background of rank one transformations from the spectral point of view 
	the reader can refer to \cite{elAbdal3}, \cite{AbPaPr} and \cite{Prikhodko}. 
	}
are called {\it spacers\/} between subsequent subintervals $(\om_n(j),\om_n(j)+h_n)$ 
(recall that the edge points  are topologically identified). 

\begin{figure}[th]
  \centering
  \unitlength=1mm
  \includegraphics{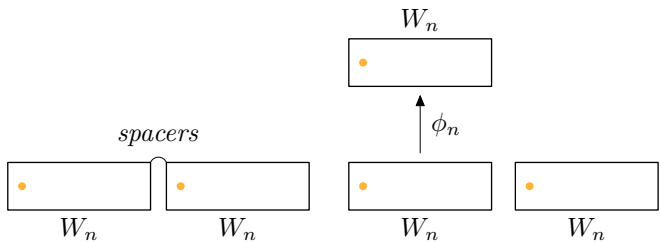} 
  \caption{Construction of a rank one flow: definition of the projection $\phi_n$} 
  \label{fRankOneFlowPhiN}
\end{figure}

Let us define the phase space $X$ of the flow to be the inverse limit of the spaces $(I_n,\cB_n)$ 
endowed with the Borel $\sigma$-algebra $\cB_n$ with respect to the maps~$\phi_n$, i.e.\ set 
\begin{equation}
	X = \{\x = (x_1,x_2,\dots,x_n,\dots) \where x_n \in I_n,\ \phi_n(x_{n+1}) = x_n\}. 
\end{equation}
The following condition ensures the correctness of the construction:
\begin{equation}
	\prod_{n=1}^\infty \frac{h_{n+1}}{q_nh_n} < \infty. 
	\label{eRankOneCorrectness}
\end{equation}
Namely, if \eqref{eRankOneCorrectness} is satisfied then there exist measures 
${\mu_n = (1-\gamma_n)\la_n + \gamma_n\delta_0}$, ${\gamma_n \to 0}$ as ${n \to \infty}$ 
and ${d\la_n = h_n^{-1}dt}$, such that ${\phi_n^* \mu_{n+1} = \mu_n}$, 
and we can define the measure $\mu$ on the limit space $X$ coinciding with $\mu_n$ 
after projecting to~$I_n$
(see \cite{Prikhodko} for further technical details). 

Now let us define the map $T$ on the space $(X,\mu)$ as follows. 
Let us fix ${t \in \Set{R}}$. For almost every point ${\x = (x_1,x_2,\dots,x_n,\dots)}$ in $X$ 
the following is true: ${|t| < x_n < h_n-|t|}$ starting from some index~$n_0$. We~define 
$Tx$ by the relation 
\begin{equation}
	(Tx)_{n=n_0}^\infty = (x_{n_0}+t,\dots,x_n+t,\dots) 
\end{equation}
and complete the sequence of coordinates $(Tx)_n$ for indexes $n$ smaller than $n_0$ 
using the fundamental relation ${\phi_n((Tx)_{n+1}) = (Tx)_n}$. 
It can be easily verified that $T$ is an invertible measurable transformation on~$X$ 
preserving the measure~$\mu$.

\subsection{Flows with simple Lebesgue spectrum and Littlewood polynomials}

\begin{thm}[see \cite{LebesgueFlows}]
\label{thSimpleLebesgueSpR}
There exist rank one flows with simple Lebesgue spectrum. 
\end{thm}

The principal analytic argument that we can use to find a rank one flow with Lebesgue spectrum%
\footnote{
	Remark that any rank one transformation has simple spectrum~\cite{LemEncycloSpTh} 
	and the same is true for any rank one flow. 
	} 
is the {\it flatness phenomenon\/} for the class of polynomials on the group~$\Set{R}$ 
\begin{equation}
	\cM^{\Set{R}} = \Bigl\{ \cP(\tau) = \frac1{\sqrt{q}} \sum_{y=0}^{q-1} e^{2\pi i\, \tau\om(y)} \where
		\om(y) < \om(y+1),\ \om(y) \in \Set{R},\ q \ge 2 \Bigr\} 
\end{equation}
called {\it Littlewood polynomials\/} with coefficients in $\{0,1\}$. 
This question goes back to the famous work due to J.\,Littlewood \cite{Littlewood} 
(see also~\cite{ErdelyiLittlewoodType02}) 
as well as investigations on 
Hardy--Littlewood series \cite{MaryWeissOnHLSeries}. 
The~Littlewood's hyposesis on flat polynomials is asking whether one can find 
a {\it unimodular\/} polynomial 
\begin{equation}
	P(z) = \frac1{\sqrt{N}} \sum_{n=0}^{N-1} a_n z^n, \qquad |a_n| = 1, \qquad N \ge 2, 
\end{equation}
such that $P(z)$ is $\eps$-ultraflat on the unit circle in the complex plane for any given~$\eps\:$? 

A~complex polynomial $P(z)$ is called $\eps$-{\it ultraflat\/} if 
\begin{equation}
	\max_{z \in S^1} \:\Bigl|\, |P(z)|-1 \,\Bigr| < \eps, \qquad S^1 = \{z \in \Set{C} \where |z| = 1\}. 
\end{equation}
This question was answered positively by J.-P.\,Kahane \cite{Kahane80}, 
though, the problem of flatness in the class of polynomials with coefficients in $\{-1,+1\}$ 
as well as in the class $\cM^{\Set{Z}}$ of polynomials with coeffitients in $\{0,1\}$ 
is wide open 
(for references and discussion see \cite{AbPaPr}, \cite{Downar}, 
	\cite{ErdelyiLittlewoodType02}, \cite{Guenais}, \cite{LebesgueFlows}). 
It is shown in paper \cite{LebesgueFlows} that in contrast to the classical flatness problem 
if $\cM^{\Set{Z}}$ that looks rather difficult 
and no reasonable arguments are known for the answer to be ``yes'' or ``no'', 
if we pass to the class $\cM^{\Set{R}}$ of polynomials on~$\Set{R}$ the answer in ``yes'' 
in $L^1$-sense on any compact subset of~$\Set{R}$ and one can find explicit examples of flat sums. 

\begin{defn}
Let us call a {\it window\/} a set of two symmetric intervals isolated both from $0$ and infinity: 
\begin{equation}
	G = (-b,-a) \cup (a,b), \qquad 0 < a < b.  
\end{equation}
\end{defn}

\begin{defn}
We say that a polynomial $\cP(t)$ on the real line~$\Set{R}$ 
is $\eps$-{\it flat\/} in~$L^p(G)$ (or $L^p(G)$-$\eps$-{\it flat\/}) if 
\begin{equation}
	\bigl\|\: |\cP(\tau)| - 1 \:\bigr\|_{L^p(G)} < \eps.  
\end{equation}
\end{defn}

\begin{thm}[see \cite{LebesgueFlows}]
\label{thOnFlatPolyR}
For any ${\eps > 0}$ and a window $G$ there exists a polynomial $\cP(t)$ 
which is $\eps$-flat both in $L^1(G)$ and $L^2(G)$ and satisfies estimate ${ |\cP(t)| \le M }$ 
for ${t \in G}$ with a global constant~$M$. 
\end{thm}

It is interesting to see that the flat polynomials in theorem~\ref{thOnFlatPolyR} 
can be represented in an explicit way.

\begin{thm}[see \cite{LebesgueFlows}]
\label{thOnFlatPolyRExplicit}
Let us fix a window $G$ and some precision ${\eps > 0}$. 
There exists ${m > 0}$ and ${\beta_0 > 0}$ such that for any ${\beta \le \beta_0}$, 
${\beta^{-1} \in \Set{N}}$, there exists an infinite sequence ${\bar q = (q_j)_{j=1}^\infty}$ 
of polynomial degrees generating $L^1$-$\eps$-flat polynomials on~$G$ 
\begin{equation}
	\cP_q(\tau) = \frac1{\sqrt{q}} \sum_{y=0}^{q-1} e^{2\pi i\, t\om(y)}
	\quad \text{with} \quad 
	\om(y) = m\frac{q}{\beta^2} e^{\beta y/q}.  
	\qquad 
	q \in \bar q. 
\end{equation}
\end{thm}

\begin{rem}
The only feature which is hard to control when choosing parameters $m$, $\beta$ and $q$ 
in theorem~\ref{thOnFlatPolyRExplicit} is the sequence $q_j$. Indeed, 
the sequence $q_j$ is the rarer the smaller $\eps$ and the more spacious window $G$ we~take. 
Here we use the term {\it spaciousness\/}%
\footnote{
	This value is connected to the length of the window in the
	logarithmic scale: ${\log b - \log a}$.
} 
of the window $G$ for the fraction~$b/a$. 
Suprisingly the value of the \uwave{$\,$smallest}\relax\ proper~$q$ is strongly related to the 
{\em diophantine properties\/} of the vector 
\begin{equation}
	{\boldsymbol v} = (\log 2,\log 3,\log 4,\dots,\log(K+1)). 
\end{equation}
More exactly, we consider the line parallel to ${\boldsymbol v}$ in the torus $\Set{T}^K$ 
starting from $0$ and study its first return time to some $\eps$-neighbourhood of the zero point. 
And the value of this return time is connected with the complexity (in particulr, the degree) 
of the polynomial $\cP(\tau)$. To explain  this phenomenon  we should mention that 
the frequency function $\om(y)$ in the proof of theorem~\ref{thOnFlatPolyRExplicit} 
is considered a~Hamiltonian of a~free quantum particle moving on the torus~$\Set{T}$, 
and the small parameter $\beta q^{-1}$ measures deviation according to the classical quadratic Hamiltonian 
($y$~is~the momentum of the particle), 
\begin{equation}
	\om(y) = {\color{green} \frac{q}{\beta^2} + \frac{y}{\beta}} + {\color{blue} \frac{y^2}{2q}} + 
		{\color{BurntOrange} \beta\frac{y^3}{6q^2} +\dots, }
\end{equation}
where $m = 1$, for example. Actually, that is exactly $\beta^{-1}q$, the value playing role of 
a~time in the dynamical system ${\dot Y = {\boldsymbol v}(Y)}$ on the torus $\Set{T}^K$ 
participating in the construction of flat polynomials with the frequency function 
${\om = \frac{q}{\beta^2}e^{\beta y/q}}$. 
\end{rem}

The following theorem generalizing lemma~\ref{thOnFlatPolyRExplicit} is of special status 
concerning the content of this paper. It cannnot used to improve the investigation of Riesz products 
on~$\Set{R}$, but it is applied in the case of rank one $\Set{R}^d$-actions 
(see the proof of theorem~\ref{thmSimpleLebSpRd} and lemma~\ref{lemLOneEstNearZero}). 

\begin{thm}\label{thmLOneCompactFlatness}
For any compact set $K$ in $\Set{R}$ and ${\eps > 0}$ there exists a polinomial 
in the class~$\cM^{\Set{R}}$ which is $\eps$-flat in $L^1(K)$. 
\end{thm}

Observe that the statement of this lemma is false in~$L^2$, and concerning $L^1$-flatness 
it is not known can we find a polynomial which is {\em globally\/} $L^1$-flat on~$\Set{R}\:$?

\subsection[Riesz products]{Generalized Riesz products and spectral measures of rank one flows}

In view of the forthcoming discussion of rank one $\Set{R}^d$-actions 
we consider in detains the proof of theorem~\ref{thSimpleLebesgueSpR} in the one-dimensional case. 
The concept of {\em generalized Riesz product\/} in the scope of rank one dynamical systems 
goes back to paper~\cite{Bourgain} by J.\,Bourgain. 
In this paper the measure of maximal spectral type $\sigma$ is claculated for 
the mixing rank one constructions introduced by D.\,Ornstein~\cite{O}. 
The measure $\sigma$ is represented in a~form of Riesz product (converging in weak sense) 
\begin{equation}\label{eRieszProduct}
	\sigma = \prod_{n=1}^\infty |P_n(z)|^2, \qquad 
	P_n(z) = \frac1{\sqrt{q_n}} \sum_{k=0}^{q_n-1} z^{\om_n(k)} 
\end{equation}
and it is discovered that this measure in purely singular with probability~$1$ 
(see also \cite{elAbdal1,elAbdal3,AbPaPr}). 

It~is important to mention the deep connection of this approach with the classical problem 
on investigation of purely singular Rajchman measures on the unit segment $[0,1]$, 
in particular, the famous question due to Rafa\"el Salem on the Minkowskii question mark function 
and his work on strictly increasing singular functions on $[0,1]$ with fast correlation decay 
(see \cite{SalemOnStrMonSingF,SalemOnSetsUniqueness} and~\cite{EliseThierryValery}). 
In order to illustrate this connection let us mention that in the most cases the spectral measures 
of rank one dynamical systems are investigated using different variations around 
the Riesz product technique, though, it is not known exactly how fast the Fourier coefficients 
\begin{equation}
	c_n = \int_0^1 e^{2\pi i\, nx} \,d\sigma(x) 
\end{equation}
can decay for the spectral measures $\sigma$ of rank one dynamical systems? At the same time, for a class 
of local rank one ergodic transformations (see~\cite{IcePaperI}) one can observe 
the extremal rate of power decay 
\begin{equation}
	c_n = O(n^{-1/2+\eps}) 
\end{equation}
for all ${\eps > 0}$, similar to R.\,Salem's examples of purely singular measures on~$[0,1]$. 
Notice that faster power decay (with some additional requirements) would ensure 
the absolute continuity of the spectral measure~$\sigma$. 
Nevertheless, no deducion can be made from such kind of information about the decay of~$c_n$, 
whether it has a Lebesgue component or not?

\begin{conj}
Suppose that a sequence of tower partitions $\xi_n$ is fixed for a rank one transformation 
such that $\xi_{n+1}$ refines $\xi_n$ and any measurable set $A$ can be approximated by a sequence 
of $\xi_n$-measurable sets~$A_n$. Then for {\it any\/}\footnote{Non-zero and with zero mean.} 
$\xi_{n_0}$-measurable function $f$ 
the sequence of the Fourier coefficients $c_n$ for the spectral measure $\sigma_f$ satisfies
\begin{equation}
	\limsup_{n \to \infty} \frac{\log|c_n|}{n} \ge -\frac12. 
\end{equation}
\end{conj}

Enclosing the discussion around analytical properties of the spectral measures 
generated by rank one dynamical systems let us remark that if we like 
construct a Lebesgue component in the spectrum of some rank one system, 
the only approach discussed in the literature (both for transformatios and flows) 
is the use of Littlewood-type flat polynomials. THough, hypothetically 
it could happen that $P_n(z)$ are not flat, but the Riesz product \eqref{eRieszProduct} 
converges to a~measure with an absolutely continuous component or even Lebesgue measure. 
In this connection we should mention that the following question 
concerning the spectral type of rank one transformations is still open. 

\begin{quest}\label{questRankOneSingSp}
For any rank one transformation the measure of maximal spectral type $\sigma$ 
is singular with respect to the Lebesgue measure? 
\end{quest}

\begin{quest}[S.\,Banach]\label{BanachProblem}
Is it possible to find an invertible measure preserving transformation 
with the Lebesgue spectrum of multiplicity one? 
\end{quest}

Therefore a rank one transformation is a candidate to the positive answer 
to the well-known question~\ref{BanachProblem} due to Stephan Banach 
(see \cite{UlamBook}, \cite{Kir67}, \cite{LemEncycloSpTh} and~\cite{LebesgueFlows}), 
there are no obstacles both for question~\ref{questRankOneSingSp} and for 
question on existence of Littlewood type flat polynomials in the class $\cM$ to be false. 
At~the same time, both questions, the flatness in the class $\cM^{\Set{R}}$ 
and Banach question, have positive answer for the group~$\Set{R}$, and 
our purpose in this paper is to discuss simple extensions of this phenomenon to 
the larger class of group actions including a class of rank one $\Set{R}^d$-actions.

\medskip
{\it Proof of theorem~\ref{thSimpleLebesgueSpR}.}
Let us consider a function ${f \Maps X \to \Set{C}}$, ${f \in L^2(X,\mu)}$, which is measurable 
with respect to $\sigma$-algebra $\cB_n$. Such function can be represented in the form
\begin{equation}
	f(\x) = f_{(n_0)}(x_{n_0}), 
\end{equation}
where $x_{n_0}$ is the $n_0$-th coordinate of a point ${\x \in X}$. 
The function $f_{n_0}$ as well as any measurable function on~$I_n$ can be lifted 
to the upper levels in accordance with the relation 
\begin{equation}
	f_{(n+1)}(x_{n+1}) = f_{(n)}(\phi_n(x_{n+1})). 
\end{equation}
Here and in the sequel we consider the functions $f_{(n)}(x_n)$ 
like function on the real line~$\Set{R}$ 
letting $f_{(n)}$ be~zero outside ${I_n = [0,h_n]}$. 
Let us define the following correlation functions:
\begin{gather}
	R(t) = \scpr<T^tf,f>, \\
	R_n(t) = (1-\gamma_n) \frac1{h_n} \int_0^{h_n} f_{(n)}(t+x_n)\,\OL{f_{(n)}(x_n)} \,dx_n, \quad n \ge n_0, 
\end{gather}
and recall that the equivariant measure $\mu_n$ on the $n$-th level of the construction 
of the rank one flow equals ${(1-\gamma_n)\la_n + \gamma_n\df_0}$ and 
${1-\gamma_n}$ is the total measure of the $n$-th tower in the cutting-and-stacking construction, 
\begin{equation}
	1-\gamma-n = \mu(U_n), \quad \text{where} \quad U_n = \{\x \in X \where 0 < x_n < h_n\}. 
\end{equation}

Without loss of generality for our purposes it suffies to explore functions 
satisfying the following requirements:
\begin{itemize}
	\item ${|f(\x)| \le M_0}$; 
	\item ${\|f\| = 1}$ in $L^2(X,\mu)$; 
	\item ${f(\x) = 0}$ outside the tower $U_n$. 
\end{itemize}
From this point we assume that these conditions are satisfied. 
Notice that ${R(0) = \|f\|^2 = 1}$ and we can also check that ${R_n(0) = 1}$. 

\begin{lem}
$R_n(0) = 1$ for any $n \ge n_0$. 
\end{lem}

\begin{proof}
The lemma follows from the observation that $f_{(n+1)}$ ``sits inside'' any tower $U_n$ 
with ${n \ge n_0}$, whenever ${\supp f_{(n_0)} \subseteq U_{n_0}}$ (up to zero measure set) 
for some starting index~$n_0$. Thus, for any $\cB_n$-measurable function 
${\psi(\x) \in L^1(X,\mu)}$, zero on~${X \sms U_n}$, which is associated with 
a function $\psi_{(n)}$ on~$[0,h_n]$ we have
\begin{equation}
	\int_X \psi(\x) \,d\mu = \frac{1-\gamma_n}{h_n} \int_0^{h_n} \psi(x) \,dx. 
\end{equation}
and the statement follows from the identity ${R_n(0) = \|\psi_{(n)}\|^2}$, 
since $R_n(t)$ is the correlation function for the shift action
\begin{equation}
	S^t \Maps \Set{R} \to \Set{R} \Maps x \mapsto t + x, 
\end{equation}
where $\|\psi_{(n)}\|$ is the norm in $L^2(\Set{R},\,(1-\gamma_n)h_n^{-1}dx)$. 
\end{proof}

Since the sequence towers $U_n$ approximate the $\sigma$-algebra of our phase space $(X,\mu)$ 
and $U_n$ covers most part of~$X$, i.e.\ ${\mu(U_n) \to 1}$, 
the functions $R_n(t)$ asymptotically close to~$R(t)$. 

\begin{lem}
If $0 < |t| < h_n$ then ${|R_n(t) - R(t)| \le \gamma_n + |t|/h_n}$. 
\end{lem}

\begin{proof}
Indeed, integrating the product $f_{(n)}(t+x_n)\,\OL{f_{(n)}(x_n)}$ in $R(t)$ 
we cannot control the influence of the set of measure $\gamma_n$ outside~$U_n$ 
and the part of the tower $U_n$ not covered by the levels that fit into the overlapping 
${I_n \cap (I_n+t)}$. Suppose that ${t > 0}$, then 
\begin{multline}
	R(t) = \int_{U_n} T^t f(\x)\,\OL{f(\x)} \,d\mu(\x) + \int_{X \sms U_n} T^tf(\x)\,\OL{f(\x)} \,d\mu(\x) = \\ 
	= \frac{1-\gamma_n}{h_n} \int_0^{h_n-t} f_{(n)}(t+x_n)\,\OL{f_{(n)}(x_n)} \,dx_n +
		\int_{U_n|_{(h_n-t,h_n)} \:\cup\: (X \sms U_n)} T^t f(\x)\,\OL{f(\x)} \,d\mu(\x), 
\end{multline}
where 
\begin{equation}
	U_n|_J \stackrel{def}{=} \{\x \in X \where x_n \in J\}, 
\end{equation}
and
\begin{equation}
	R_n(t) = \frac{1-\gamma_n}{h_n} \int_0^{h_n-t} f_{(n)}(t+x_n)\,\OL{f_{(n)}(x_n)} \,dx_n, 
\end{equation}
hence, taking into account the requirement ${|f| \le 1}$, 
\begin{equation}
	|R_n(t) - R(t)| \le \mu\bigl(U_n|_{(h_n-t,h_n)} \:\cup\: (X \sms U_n)\bigr) \le t/h_n + \gamma_n. 
\end{equation}
The case ${t < 0}$ is symmetrical. 
\end{proof}

As a direct corollary we get the following statement. 

\begin{lem}
For a fixed $\cB_{n_0}$-measurable function $f(\x)$ satisfying the conditions stated above 
then the correlation functions $R_n(t)$ converges pointwise to $R(t)$, i.e.\ 
\begin{equation}
	\forall t \in \Set{R} \quad R_n(t) \to R(t). 
\end{equation}
\end{lem}

The next lemma is the common property of measurable $\Set{R}$-actions on a Lebesgue space 
(see~\cite{KSF}). 

\begin{lem}
Any measurable flow $T^t$ is continuous, i.e.\ ${R_f(t) \to 0}$ as ${t \to 0}$ 
for any ${f \in L^2(X,\mu)}$, where ${R_f(t) = \scpr<T^tf,f>}$. 
\end{lem}

Now using the well-known Levy's lemma we can deduce the convergence of 
the corresponding probability distributions $\sigma_n$, where ${{\Hat\sigma}_n = R_n(t)}$. 

\begin{lem}[Levy]
\label{lemLevy}
Given a sequence of probability distributions $\nu_n$ on $\Set{R}$ as well as 
a distribution $\nu^*$, if the characteristic functions%
\footnote{
	It would be more rigorous to use a sign ${\check\nu(t) = \int_{\Set{R}} e^{2\pi i\, tx} \,d\nu(x)}$ 
	for the inverse Fourier transform, but for simplicity we use the same symbol ``hat'' both for 
	direct $\widehat{r(t)}$ and inverse Fourier transform $\Hat\nu(t)$, since 
	it is evident from the context which kind of transform is used. 
	}
$\Hat\nu_n(t)$ converge pointwise to the characteristic function $\Hat\nu^*$ and 
the limit funtion $\Hat\nu^*$ is continuous at zero than $\nu_n$ converges weakly to $\nu^*$ 
with respect to the space $C_b(\Set{R})$ of bounded continuous functions (with the $C$-norm), i.e.\ 
\begin{equation}
	\forall \,\phi(x) \in C_b(\Set{R}) \quad \int \phi(x) \,d\nu_n \to \int \phi(x) \,d\nu. 
\end{equation}
\end{lem}

Let us make the following remark. It is important in Levy's lemma that we know in advance that 
the limit function ${r(t) = \lim_{n \to \infty} \Hat\nu_n(t)}$ coincides with the characteristic 
function of some probability distribution $\nu^*$. 

\begin{lem}
Any function ${\Hat R}_n(\tau)$ is a density of a positive measure $\sigma_n$ on~$\Set{R}$ and 
${\|\sigma_n\| = 1}$. 
\end{lem}

\begin{proof}
The first statement follows from the following explicit form of the correlation function 
for the shift action $S^t$:
\begin{equation}
	R_n(t) = f_{(n)} \smartconv{n} \tilde f_{(n)}, 
\end{equation}
where ${\tilde f_{(n)}(x) = \OL{f_{(n)}(-x)}}$ 
and the convolution is defined using the same normalizing 
multiplier like the $L^2$-norm on~$I_n$:
\begin{equation}
	(f \smartconv{n} g)(t) = \frac{1-\gamma_n}{h_n} \int_{\Set{R}} f(t-x)\,g(x) \,dx. 
\end{equation}
Hence, 
\begin{equation}
	{\Hat R}_n(\tau) = (f_{(n)} \smartconv{n} \tilde f_{(n)})\,\Hat{\hbox{\,}}\relax = 
		\frac{1-\gamma_n}{h_n}\,|\Hat f_{(n)}|^2(\tau), 
\end{equation}
and, finally, for the measure $\sigma_n = {\Hat R}_n(\tau)\,d\tau$ we have 
${\|\sigma_n\| = R(0) = 1}$. 
\end{proof}

Let us note that the symbol $\Hat f$ means the ordinary Fourier transform, 
\begin{equation}
	\Hat f(\tau) = \int_{\Set{R}} e^{2\pi i\, \tau\,t} f(t) \,dt, 
\end{equation}
and with the following notation for the normalized Fourier transform 
\begin{equation}
	\F_n[f](\tau) = \sqrt{\frac{1-\gamma_n}{h_n^{-1}}} \cdot {\Hat f}(\tau). 
\end{equation}
we get the following smart representation for ${\hat R}_n(t)$:
\begin{equation}
	{\Hat R}_n(\tau) = \bigl| \F[f]_{(n)} \bigr|^2(\tau). 
\end{equation}

\begin{lem}\label{lemWeakConvOfSigmaN}
The sequence $\sigma_n$ converges weakly to the spectral measure $\sigma$. 
\end{lem}

\begin{proof}
From the spectral theorem we know that $R(t)$ is the (inverse) Fourier transform 
of the spectral measure~$\sigma$, hence, we can apply Levy's lemma to the sequence~$\sigma_n$ 
taking into account continuity at zero of~$R(t)$. 
\end{proof}

\begin{rem}
This lemma opens a set of non-trivial effects. 
For the sake of the forthcoming analytical investigation of the measure $\sigma$ 
it is important to know that $\sigma_n$ {\em a~priori\/} converges to~$\sigma$. Though, it is hard 
to follow the sequence of densities ${\Hat R}_n(\tau)$, for example, if we try 
to examine the local structure of ${\Hat R}_n(\tau)$ on an interval $(\tau_1,\tau_2)$. 
The behavior of the densities ${\Hat R}_n(\tau)$ could be very complicated. 

In the case of a rank one transformation it is shown that $\sigma$ can be calculated, 
in a sense, {\em directly}, 
in the form of generalized Riesz product $\prod_n |P_n(z)|^2$ (see~\cite{elAbdal1,AbPaPr,Bourgain}). 
In other words, there is no need to follow the sequence ${\Hat R}_n(z)$ for 
an individual function~$f(\x)$. On the contrary, the $\Set{R}$-case is more intriguing, 
and we need to apply additional restrictions to the huge variety of $\Set{R}$-based Riesz product. 
In our case the initial density ${\Hat R}_{(n_0)}(\tau)$ plays the role 
of regularizing multiplier in the Riesz product. Though, a~priory there is no obvious way 
to eliminate this multiplier and to extract some purely analytic description of 
the global convergence for the product $\prod_n |\cP_n(t)|^2$ on the real line~$\Set{R}$. 
This effect with regard to our case leads to the following phenomenon. 
In fact, we know a~priory that ${\sigma_n \to \sigma}$ weakly on~$\Set{R}$, but we can only 
control the structure of the limit distribution on any window $G_{a,b} = {(-b,-a) \cup (a,b)}$, ${a > 0}$, 
and a part of the mass in~$\sigma_n$ can ``escape'' into the boundary set 
\begin{equation}
	\cF = \Set{R} \sms \bigcup_{b > a > 0}^\infty G_{a,b} = \bigl\{ 0 \bigr\} 
\end{equation}
which is exactly the set containing zero point $\{0\}$. 
%
In other words, the limit measure $\sigma$ can have an atom at zero 
(this is the only possible measure on a one-point set), and to see that 
$\sigma$ is absolutely continuous we need to apply second ergodic argument, namely, 
the ergodicity of our rank one flow $T^t$ that ensures that $\sigma$ has no atom at zero%
\footnote{
	As usual, we consider the Koopman operator $\hat T$ in the subspace of functions with zero mean.
	}. 
Thus, the convergence of~$\sigma_n$ is established using several argumets 
coming from the ergodic theory background, 
and the general question on the global convergence of Riesz products on~$\Set{R}$ 
is an object for special investigations. 
\end{rem}


Now we pass to the second logical part of the proof that can be entitled: 
investigation of the limit distribution $\lim_{n \to \infty} \sigma_n$. 

\begin{lem}
The densities ${\Hat R}_n(\tau)$ can be calculated using the following reccurent relation 
\begin{equation}
	{\Hat R}_{n+1}(\tau) = {\Hat R}_n(\tau) \cdot |\cP_n(\tau)|^2, 
\end{equation}
where
\begin{equation}
	\cP_n(\tau) = \frac1{\sqrt{q_n}} \sum_{y=0}^{q_n-1} e^{2\pi i\, \tau\,\om_n(y)}. 
\end{equation}
Thus, 
\begin{equation}
	{\Hat R}_{N+1}(\tau) = {\Hat R}_{n_0}(\tau) \: \prod_{n=n_0}^N |\cP_n(\tau)|^2. 
\end{equation}
\end{lem}

\begin{proof}
Since 
\begin{equation}
	R_n(t) = (f_{(n)} \smartconv{n} \tilde f_{(n)})(t), 
\end{equation}
it is sufficient to look at 
the corresponding recurrent formula for~$f_{(n)}$. Indeed, it can be easily seen that 
\begin{equation}
	f_{(n+1)} = f_{(n)} \conv (\df_{\om_n(0)} + \dots + \df_{\om_n(q_n-1)}), 
\end{equation}
hence, passing to Fourier transforms, we have 
\begin{equation}
	{\Hat f}_{(n+1)}(\tau) = 
		{\Hat f}_{(n)}(\tau) \cdot \sum_{y=0}^{q_n-1} e^{2\pi i\, \tau\,\om_n(y)} = 
		{\Hat f}_{(n)}(\tau) \cdot \sqrt{q_n}\:P_n(\tau), 
\end{equation}
and
\begin{equation}
	{\Hat R}_{n+1}(\tau) = 
		\frac{1-\gamma_{n+1}}{h_{n+1}}\: |{\Hat f}_{(n+1)}(\tau)|^2 =
		\frac{1-\gamma_{n+1}}{h_{n+1}}\: |{\Hat f}_{(n)}(\tau)|^2 \cdot q_n\:|P_n(\tau)|^2 = 
		{\Hat R}_n \cdot |P_n(\tau)|^2. 
\end{equation}
Here we use the following fundamental relation 
\begin{equation}
	\frac{h_{n+1}}{q_n h_n} = \frac{1-\gamma_{n+1}}{1-\gamma_n}, 
\end{equation}
implying 
\begin{equation}
	\frac{1-\gamma_{n+1}}{h_{n+1}} \cdot q_n = (1-\gamma_{n+1}) \cdot \frac{1-\gamma_n}{h_n\,(1-\gamma_{n+1})} 
		= \frac{1-\gamma_n}{h_n} 
\end{equation}
and ${\Hat R}_{n+1}(\tau) = 
	(1-\gamma_n)\,h_n^{-1}\:|{\Hat f}_{(n)}(\tau)|^2 \cdot |P_n(\tau)|^2 = 
	{\Hat R}_n(\tau) \cdot |P_n(\tau)|^2$. 
\end{proof}

Our next purpose is to analyze the convergence of the densities ${\Hat R}_n(\tau)$ 
on a window $G_{a,b}$, ${a > 0}$, separated from zero. 
We are going to prove the following common lemma 
and to apply this observation to the Riesz product 
for a sequence of flat polynomials $\cP_n(\tau)$. 

\begin{lem}
\label{lemWeakConvergesOfSigmaN}
Consider a sequence of positive probability distributions $\nu_n$ on~$\Set{R}$ 
having regular densities ${\rho_n \in L^1(\Set{R})}$. Suppose that $\nu_n$ 
converges weakly to a probability distribution $\nu^{**}$ and, at the same time, 
for any segment $[a,b]$, ${a > 0}$, the functions $\rho_n|_{[a,b]}$ converges in $L^1[a,b]$. 
Then the limit distribution splits into a sum of an absolutely continuous measure 
and an atom at zero,
\begin{equation}
	\nu^{**} = \a\,\df_0 + \Phi(\tau)d\tau, \quad \text{where} \quad \a = \nu(\{0\}), 
\end{equation}
moreover, for any segment $[a,b]$, ${a > 0}$, the following holds in $L^1[a,b]\:$:
\begin{equation}
	\lim_{n \to \infty} \rho_n|_{[a,b]} = \Phi|_{[a,b]}, 
\end{equation}
and $\|\Phi\|_1 = \nu^{**}(\Set{R} \sms \{0\}) = 1 - \a$. 
\end{lem}

The idea of this lemma is very simple. We control the convergence of $\sigma_n$ 
on any window $G_{a,b}$, in fact, $\sigma_n$ converge even in strong sense, 
but some mass can ``escape'' outside all windows $G_{a,b}$, ${0 < a < b}$. 
Thus, generally we must take into account the atom at zero. 

\begin{proof}
First, notice that in our consideration we can omit a set of points which are 
very far from zero, i.e.\ using a simple fact of real anaysis we can find ${L > 0}$ 
such that 
\begin{equation}
	\nu^{**}(\Set{R} \sms [-L,L]) < \eps 
\end{equation}
for some fixed ${\eps > 0}$, since $\nu^{**}$ is a $\Sigma$-finite Borel measure on~$\Set{R}$. 
Then, roughly speaking, we can restrict ourselves to 
the compact set $[-L,L]$ and identify measures with the corresponding 
bounded linear functionals on $C[-L,L]$. 
We also choose ${\delta > 0}$ such that 
\begin{equation}
	\nu^{**}\bigl( (-\delta,0) \cup (0,\delta) \bigr) < \eps. 
\end{equation}
Now let us consider a function ${\varphi \in C_b(\Set{R})}$ and split $\varphi$ 
into the sum
\begin{equation}
	\varphi = \varphi_0 + \varphi_1 + \varphi_2, \qquad \varphi_j(\tau) = r_j(\tau)\varphi(\tau), \qquad 
	r_j \in C_b(\Set{R}), 
\end{equation}
where 
\begin{gather}
	r_0(\tau) + r_1(\tau) + r_2(\tau) \equiv 1, \qquad |r_j(\tau)| \le 1, 
	\\
	\supp \,r_0 = [-\delta,\delta], \qquad 
	\supp \,r_1 = G_{\delta/2,L+1}, \qquad 
	\supp \,r_2 = (-\infty,-L] \cup [L,\infty), 
	\\
	r_1(\tau) \equiv 1 \quad \text{if} \quad \tau \in G_{\delta,L}. 
\end{gather}
For simlicity, let us use notation 
\begin{equation}
	\scpr<\varphi,\nu> \stackrel{def}{=} \int \varphi(\tau) \,d\nu. 
\end{equation}
We know that
\begin{equation}
	\lim_{n \to \infty} \int \varphi(\tau) \,d\nu_n = 
		\scpr<\varphi,\nu^{**}> = 
		\scpr<\varphi_0,\nu^{**}> + 
		\scpr<\varphi_1,\nu^{**}> + 
		\scpr<\varphi_2,\nu^{**}>, 
\end{equation}
and
\begin{equation}
	\left|
		\scpr<\varphi_0,\nu^{**}> - \a \varphi(0)
	\right|
	\le \eps \cdot \|\varphi\|_\infty, 
	\qquad
	\left|
		\scpr<\varphi_2,\nu^{**}> 
	\right|
	\le \eps \cdot \|\varphi\|_\infty, 
	\label{eEstNuStarStarJ}
\end{equation}
hence, 
\begin{equation}
	\Bigl|
		\scpr<\varphi,\nu^{**}> - 
		\scpr<\varphi_1,\nu^{**}> - \a\varphi(0) 
	\Bigr|
	< 2\eps \cdot \|\varphi\|_\infty. 
\end{equation}
At the same time, for any component $\varphi_j$ we have 
${ \scpr<\varphi_j,\nu_n> \to \scpr<\varphi_j,\nu^{**}> }$ as ${n \to \infty}$, 
hence, passing to the limit we see that 
\begin{equation}
	\scpr<\varphi_1,\nu^{**}> = 
		\lim_{n \to \infty} \int_{\delta/2}^{L+1} \varphi(\tau) \,\rho_n(\tau) \,d\tau 
	= \int_{\delta/2}^{L+1} \varphi(\tau) \,\Phi|_{[\delta/2,L+1]}(\tau) \,d\tau, 
\end{equation}
where $\Phi|_{[\delta/2,L+1]}$ is the limit of $\rho_n$ in $L^1[\delta/2,L+1]$. 
It can be easlily checked that $L^1$-limit of $\rho_n$ does not depend on the choice 
of a window $G_{a,b}$, ${a > 0}$, so that $\Phi|_{[a,b]}$ is the restriction to~$[a,b]$ 
of some locally $L^1$-function $\Phi$ on~$\Set{R}$. More exactly, 
if ${\rho_n \to \Phi_1}$ in $L^1[a,b]$ and ${\rho_n \to \Phi_2}$ in $L^1[a',b']$ 
for a wider segment ${[a',b'] \supseteq [a,b]}$ then ${\Phi_1 = \Phi_2|_{[a,b]}}$. 
Notice that since $\nu_n$ converges weakly to $\nu^{**}$ 
\begin{equation}
	\left| \int \varphi_1(\tau) \,\Phi(\tau) \,d\tau \right| = 
	\lim_{n \to \infty} |\scpr<\varphi_1,\nu_n>| = |\scpr<\varphi_1,\nu^{**}>| 
		\le \|\nu^{**}\| \cdot \|\varphi_1\|_\infty = \|\varphi_1\|_\infty, 
\end{equation}
thus, $\|\Phi|_{G_{a,b}}\|_1 \le 1$ for any window $G_{a,b}$, ${a > 0}$, 
and it evidently follows from this extimate that 
\begin{equation} 
	\Phi \in L^1(\Set{R}) \quad \text{and} \quad \|\Phi\|_1 \le 1. 
\end{equation}
Integrating all the above arguments we can deduce that the limit measure 
${\nu^{**} = \a\df_0 + \Phi(\tau)\,dt}$, and aplying it to the unit function 
we see that ${\|\Phi\|_1 = 1-\a}$. 
\end{proof}

\begin{lem}
\label{lemLOneConverges}
Consider a uniformely bounded sequence of 
non-negative continuous functions $Q_n(\tau)$ on a segment $[a,b]$ 
satisfying the following conditions 
\begin{equation}
	\forall\, \tau \in [a,b] \quad Q_n(\tau) \le M, \qquad 
	\bigl\|\, Q_n - 1 \,\bigr\|_1 \le \eps_n, 
\end{equation}
and suppose that ${ \sum_n (M^n \eps_n)^{1/2} < \infty }$ and ${M > 1}$. 
Then the Riesz product $\prod_n Q_n$ converges in~$L^1[a,b]$. 
\end{lem}

\begin{rem}
Let us observe that it is {\em not sufficient\/} to require $L^1$-$\eps-n$-flatness 
of the multipliers $Q_n$ whatever rate of decay we require for~$\eps_n$. 
Indeed, let us show, for example, that there exists a sequence of 
uniformely bounded functions $Q_n$ 
on $[0,1]$ with exponential estimate ${ \|Q_n - 1\|_1 \le \eps_n = 2^{-n} }$ such that 
$\prod_n Q_n$ converges to the delta-function~$\df_1$. We define $Q_n(\tau)$ as follows 
\begin{equation}
	Q_n(\tau) = \left\{ 
			\begin{alignedat}{2}
				&1, 	\quad		& & \text{if $0 \le \tau < 1 - 2^{-n}$}, \\ 
				&0, 					& & \text{if $1 - 2^{-n} \le \tau < 1 - 2^{-n-1}$}, \\ 
				&2, 					& & \text{if $1 - 2^{-n-1} \le \tau < 1$}  
			\end{alignedat}
		\right.
\end{equation}
Clearly, $Q_n(\tau) \approx 1$ everywhere but $2^{-n}$-small set ${ [1 - 2^{-n},\:1] }$, 
and also ${ \|Q_n - 1\|_1 = 2^{-n} }$, but 
\begin{equation}
	\prod_{n=0}^N Q_n = \left\{ 
			\begin{alignedat}{2}
				&0, 							& & \text{if $0 \le \tau < 1 - 2^{-N-1}$}, \\ 
				&2^{N+1},	\quad		& & \text{if $1 - 2^{-N-1} \le \tau < 1$}  
			\end{alignedat}
		\right. 
		\;\; \to \df_1, \quad \text{when $n \to \infty$}. 
\end{equation}
\end{rem}

\begin{quest}
An observation of positive kind: if non-negative functions $Q_n$ on $[0,1]$ 
are indepedent as random variables with respect to the Lebesgue measure 
and ${ \|Q_n - 1\|_1 < \eps_n }$, ${\sum_n \eps_n < \infty}$, 
then the product $\prod_n Q_n$ convergence in~$L^1[0,1]$. 
This idea can be applied as well to the sequence $|\cP_n(\tau)|^2$ of flat polynomials 
associated with a~rank one flow, but they are just very close to indpendence. 

How close $Q_n$ should come to satisfying the independence property 
to ensure the $L^1$-convergence of the Riesz product? 
\end{quest}

\begin{proof}[Proof of lemma~\ref{lemLOneConverges}]
%
Let us define sets 
\begin{equation}
	A_n = \bigl\{ x \where |Q_n(x) - 1| > \a_n \bigr\}, \qquad 
	\la(A_n) \le \frac{\eps_n}{\a_n}, 
\end{equation}
where ${ \alpha_n = (M^n \eps_n)^{1/2} }$ and $\la$ is the Lebesgue measure on~$\Set{R}$. 
We use Chebyshev's inequality estimating the value of~$\la(A_n)$. 
Remark that ${ \sum_n \alpha_n < \infty }$ by the conditions of the lemma. 
Let us build a code any point $x$ setting ${x_n = 1}$ if ${x \in A_n}$ and ${x_n = 0}$ otherwise, 
and let $h(x)$ be the index of the last ``$1$'' in the code of~$x$. 
Taking into account that 
\begin{equation}
	\sum_{n=1}^\infty \la(A_n) 
		\le \sum_{n=1}^\infty \frac{\eps_n}{\a_n} 
		=   \sum_{n=1}^\infty \a_n M^{-n} 
		\le \sum_{n=1}^\infty \a_n < \infty 
\end{equation}
then in force of Borel--Cantelli lemma $h(x)$ is correctly defined for almost all points~$x$. 
Set 
\begin{equation}
	B_n = \{x \where h(x) = n\}. 
\end{equation}
Since ${B_n \subseteq A_n}$, then ${\la(B_n) \le \la(A_n)}$. 
We cannot control the behaviour of mutipliers $Q_k$ on~$B_n$ for the indexes ${k \le n}$, 
but we have the global estimate for the $Q_k$, 
\begin{equation}
	\forall\,x \in B_n \quad \prod_{k=1}^N Q_k = 
		\prod_{k=1}^n Q_k \cdot \prod_{k=n+1}^N Q_k \le 
		M^n \cdot \prod_{k \:>\: n} (1 + \a_n) \le \Pi_0\,M^n, 
\end{equation}
where
\begin{equation}
	\Pi_0 = \prod_{n=1}^\infty (1 + \a_n). 
\end{equation}
Let us define the following global $L^1$-majorant for our Riesz product: 
\begin{equation}
	M(x) = \Pi_0\,M^n \quad \text{if} \quad x \in B_n 
\end{equation}
and $M(x) = \Pi_0$ if $x$ does not belong to any $B_n$. 
The function $M(x)$ is integrable (belongs to $L^1[a,b]$), 
since the following series converges: 
\begin{equation}
	\sum_{n=1}^\infty \mu(B_n) \cdot \Pi_0\,M^n \le 
	\Pi_0\, \sum_{n=1}^\infty \frac{\eps_n}{\a_n} \cdot M^n = \Pi_0\, \sum_{n=1}^\infty \a_n < \infty. 
\end{equation}
Thus, by the Lebesgue dominated convergence theorem our product $\prod_n Q_n(x)$ 
converges in $L^1[a,b]$, whenever it converges pointwise for almost all points~$x$. 
\end{proof}

\begin{lem}
In the scope of the previous lemma it is enough 
to require the following modified set of conditions:
\begin{gather}
	Q_n(x) \le M_n, \qquad 1 < M_1 < \dots < M_n < \dots, \\ 
	\a_n^2 = \eps_n \, \prod_{k=1}^n M_n, \qquad \sum_{n=1}^\infty \a_n < \infty. 
\end{gather}
\end{lem}

\begin{proof}
In fact, let us check for convergence the serieses examined in the proof of the previous lemma. 
First, we have the following estimate for Borel--Cantelli lemma: 
\begin{equation}
	\sum_{n=1}^\infty \frac{\eps_n}{\a_n} = \sum_{n=1}^\infty \a_n \prod_{k=1}^n M_k^{-1} < 
		\sum_{n=1}^\infty \a_n < \infty. 
\end{equation}
Next,
\begin{equation}
	\sum_{n=1}^\infty 
		\left( 
			\frac{\eps_n}{\a_n} \cdot \prod_{k=1}^n M_k 
		\right) 
		=
		\sum_{n=1}^\infty \a_n < \infty. 
\end{equation}
Finally, ${\sum_n \a_n < \infty}$ by the conditions of the lemma. 
\end{proof}

\begin{lem}\label{lemRieszConvRate}
Assume that the conditions of lemma~\ref{lemLOneConverges} are satisfied 
and set ${\eps_0 = \sum_n \alpha_n}$. 
Then there is a~set $U$ of measure ${b - a - \eps_0}$ such that 
\begin{equation}
	\forall\, x \in U \quad \left| \prod_n Q_n - 1 \right| < \Pi_0 - 1 \le \exp\!\left( \sum_n \a_n \right) - 1. 
\end{equation}
If we additionally require that all ${\eps_0 < 1}$ then 
${ \Pi_0 - 1 < 3 \eps_0 }$. 
\end{lem}


Let us apply lemmas \ref{lemWeakConvergesOfSigmaN} and~\ref{lemLOneConverges} 
to the Riesz product generated by the rank one flow 
with flat polynomials $\cP_n(\tau)$. We begin with the repetition 
of the main elements of the construction. 

\begin{constr}
\label{cLROF}
Let us choose an encreasing sequence of windows expanding to the whole $\Set{R}$ 
except zero point: 
\begin{equation}
	G_{a_n,b_n} \subset G_{a_{n+1},b_{n+1}}, \qquad \bigcup_n G_{a_n,b_n} = \Set{R} \sms \{0\}, 
\end{equation}   
and for each winfow $G_{a_n,b_n}$ we let us find a~flat polynomial (see~\cite{LebesgueFlows}) 
\begin{equation}
	\cP_n(\tau) = \frac1{q_n} \sum_{y=0}^{q_n-1} e^{2\pi i\, \tau \,\om_n(y)}, \qquad 
	\om_n(y) = m_n \frac{q_n}{\beta_n^2} e^{\beta_n y/q_n}, 
\end{equation}
which should be compatible, of course, with the rank one construction, in particular, 
\begin{equation}
	h_n = \frac{m_n}{\beta_n}\,(1 + o(1)), \qquad \sum_{n=1}^\infty \frac{m_n}{h_n} < \infty. 
	\label{eFormulaForHn}
\end{equation}
We can choose $\cP_n(\tau)$ in such a way that 
\begin{gather}
	\|\cP_n(\tau)|_{G_{a_n,b_n}}\|_\infty \le M_n^{1/2}, 
	\\
	\bigl\|\, |\cP_n(\tau)|_{G_{a_n,b_n}}|^2 - 1 \,\bigr\| < \eps_n \quad \text{in $L^1(G_{a_n,b_n})$}, 
\end{gather}
and $M_n \eps_n$ goes to~$0$ as fast as it is required. Assume that ${M_n > 0}$. 
Notice that all the parameters $m_n$, $\beta_n$ and $q_n$ are choosen when $\eps_n$, $h_n$, 
the window $G_{a_n,b_n}$ and all the parameters from the previous steps are fixed 
(including $\cP_k(\tau)$, ${k = 1,\ldots,n-1}$), so that we can take an appropriate 
$m_n$ and $\beta_n$ to ``cover'' the window $G_{a_n,b_n}$ and to match $h_n$ (see~\eqref{eFormulaForHn}) 
and vary $q_n$ going far towards infinity to fit any predefined precision~$\eps_n$. 
\end{constr}

\begin{rem}
According to \cite{LebesgueFlows} the flat polynomials $\cP_n(\tau)$ in construction~\ref{cLROF} 
possesses a~universal upper bound ${M \equiv M_n}$. 
\end{rem}

\begin{lem}\label{lemSimpleLebesgueSpGen}
Suppose that the following conditions on the flat polynomials $\cP_n(\tau)$ are satisfied: 
\begin{gather}
	|\cP_n(\tau)| \le M^{1/2}, 
	\\
	\bigl\|\, |\cP_n|^2 - 1 \,\bigr\|_1 \le \eps_n 
	\quad \text{and} \quad 
	\sum_{n=1}^\infty (M^n \eps_n)^{1/2} < \infty, 
\end{gather}
whenever $\tau$ ranges over a window $G_{a_n,b_n}$, 
and ${M > 1}$ is a global constant. 
Then for any $\cB_n$-measurable bounded function $f(\x)$ 
having support in the tower $U_n$ the spectral measure $\sigma_f$ 
is absolutely continuous with respect to the Lebesgue measure on~$\Set{R}$. 
\end{lem}

\begin{proof}
On the one side, in view of lemma~\ref{lemLOneConverges} the sequence of measurable functions 
\begin{equation}
	{\Hat R}_{N+1}(\tau) = |\Hat f_{n_0}|^2 \, \prod_{n=n_0}^{N} |\cP_n(\tau)|^2 
\end{equation}
converges in $L^1(G)$ for any window $G_{a,b}$, ${a > 0}$. 
On the other side, the measures $\sigma_n$ having ${\Hat R}_n(\tau)$ as the density 
converge weakly to the spectral measure $\sigma_f$ by lemma~\ref{lemWeakConvOfSigmaN}. 
Hence, in force of lemma~\ref{lemWeakConvergesOfSigmaN}
\begin{equation}
	\sigma_f = \a\df_0 + \Phi_f(\tau)\,d\tau, 
\end{equation}
where $\Phi_f \in L^1(\Set{R})$. It follows from the ergodicity of the rank one flow $T^t$ 
that $\sigma_f$ has no atom at zero, and ${\sigma_f = \Phi_f(\tau)\,d\tau}$. 
\end{proof}

The following lemma completes the proof of theorem~\ref{thSimpleLebesgueSpR}.

\begin{lem}\label{lemAbsContRieszProduct}
The maximal spectral type $\sigma$ of the rank one flow constructed with the help 
of a sequence of flat polynomials $\cP_n(\tau)$ is Lebesgue. 
\end{lem} 

\begin{proof}
It is enough to prove 
(in addition to the statement of lemma~\ref{lemAbsContRieszProduct}) 
that for any segment $[a,b]$ and ${\eps > 0}$ 
there exists a spectral measure ${\sigma_f = \Phi_f(\tau)\,d\tau}$ 
for the flow $T^t$ such that 
\begin{equation}
	\la\bigl( \{ \tau \in [a,b] \where \Phi_f(\tau) = 0 \} \bigr) < \eps. 
\end{equation}
Let us find $n_0$ and a function $f(\x)$ satisfying the following requirements:
\begin{gather}
	[a,b] \subseteq G_{a_n,b_n}, \qquad n \ge n_0, 
	\\ 
	\sum_{n=n_0}^\infty \a_n < \eps   \quad \text{with} \quad   \a_n^2 = M^n \eps_n, 
	\\
	\forall\,\tau \in G_{a_{n_0},b_{n_0}} \quad {\Hat f}_{(n)}(\tau) > 0. 
\end{gather}
Then the density of the measure $\sigma_f$ is positive on a~set of measure 
at least ${b-a-\eps}$. 
\end{proof}

\section{$\Set{R}^d$-Actions} 

In this section we extend the construction of rank one flow with simple Lebesgue spectrum 
to the case of rank one $\Set{R}^d$-actions. 

\subsection{Beginning illustration: tensor square $T^t \times T^s$} 

We would like to start with a simple illustration that helps to see 
what kind of effects we need overcome when passing to the multi-dimensional case. 
Let us assume that ${d = 2}$ and consider for a given rank one flow 
${\T = \{T^t \where t \in \Set{R}\}}$ acting on the space $(X,\mu)$ 
its tensor product ${ \T \otimes \T = \{T^t \times T^s \where (t,s) \in \Set{R}^2\} }$ 
acting on $(X \times X,\mu \times \mu)$, 
\begin{equation}
	(T^t \times T^s)(x,y) = (T^t x, T^sy). 
\end{equation}

On the one hand, the space $L^2(X \times X,\mu \times \mu)$ splits into the four invariant spaces: 
\begin{equation} 
	L^2(X \times X,\mu \times \mu) = \{\const\} \oplus 
		(\{\const\} \otimes H) \oplus (H \otimes \{\const\}) \oplus (H \otimes H), 
\end{equation} 
where ${H = \{f \in L^2(X,\mu) \where \int f \,d\mu = 0\}}$ is the subspace of functions with zero mean. 
Thus, whenever $T^t$ has Lebesgue spectrum of multiplicity one, 
its tensor square will also have simple spectrum, but its spectral type $\sigma^{(2)}$ 
will be the superposition of the Lebesgue%
\footnote{
	For the correctness we should speak about a class of finite measures 
	equivalent to the Lebesgue measure.
	} 
measure $\la^{(2)}$ on~$\Set{R}$ in the subspace $H \otimes H$ and two singular components: 
the Lebesgue measure on $X$-axis ${\df_0(\tau_2)= \la \times \df_0}$ (in $H \otimes \{\const\}$) 
and the Lebesgue measure on $Y$-axis ${\df_0(\tau_1)= \df_0 \times \la}$ (in $\{\const\} \otimes H$), 
and, of course, an~atom $\df_{(0,0)}$ corresponding the subspace $\{\const\}$. 

On the other hand, let us remark that $\T \otimes \T$ can be viewed as a result 
the cutting-and-stacking construction for rank one actions of the group $\Set{R}^2$ 
involving the~sequence of towers $U_n$ 
having square shape $I_n \times I_n$, where ${I_n = [0,h_n]}$. 
And it is interesting to find an iterpretation of the appearance 
of the singular part of~$\sigma^{(2)}$ in terms of Riesz products. 
This question helps understand better 
the structure of the Riesz product 
generated by a sequence of flat polynomials in the one-dimensional case. 

\begin{figure}[th]
	\centering
  \unitlength=1mm
  \includegraphics[height=11cm]{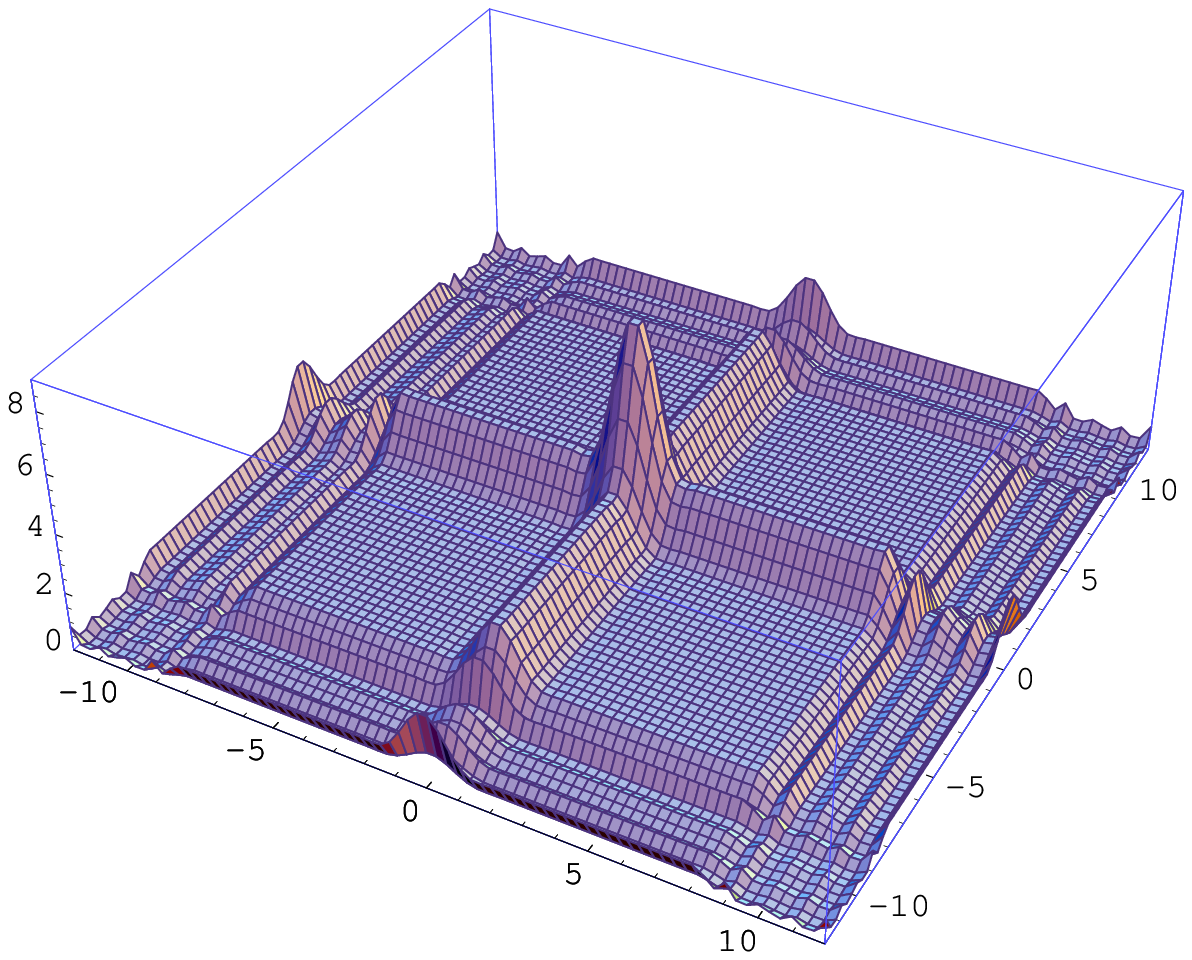} 
  \caption{Polynomials $\cP_n(\tau)$ for $\T \otimes \T$} 
  \label{fIcebergAndRank}
\end{figure}

For this purpose consider a trigonometric sum 
\begin{gather}
	\cP(\tau_1,\tau_2) = \cP_X(\tau_1)\,\cP_Y(\tau_2), 
	\\
	\cP_X(\tau_1) = \frac1{\sqrt{q}} \sum_{y_1=0}^{q-1} e^{2\pi i\, \tau_1\omega_X(y_1)}, \qquad 
	\cP_Y(\tau_2) = \frac1{\sqrt{q}} \sum_{y_2=0}^{q-1} e^{2\pi i\, \tau_2\omega_Y(y_2)}, 
	\\
	\omega_X(y) = \omega_Y(y) = m\frac{q}{\beta^2} e^{\beta y/q}. 
\end{gather}
for a single step in the rank one construction (for simplicity we omit index~$n$). 
Clearly, ${\cP = \cP_X \otimes \cP_Y}$ is just a tensor square of $\cP_X$. 
We can represent $\cP(\tau_1,\tau_2)$ in the following invariant form: 
\begin{equation}
	\cP(\tau) = \frac1{q} \sum_{y_1=0}^{q-1} \sum_{y_2=0}^{q-1} 
		e^{2\pi i\, \scpr<\tau,\omega(y)>}, 
\end{equation}
where ${\tau = (\tau_1,\tau_2)}$, ${y = (y_1,y_2)}$, and 
\begin{gather}
	\scpr<\tau,\omega(y)> = \tau_1\om_X(y) + \tau_2\om_Y(y), 
	\\
	\omega(y) = \mu\frac{q}{\beta^2} (e^{\beta y_1/q}, e^{\beta y_2/q}). 
	\label{eToDimFrequencyF}
\end{gather}
The following lemma directly follows from the one-dimensional one. 

\begin{lem}
Consider a window in $\Set{R}^2$ 
\begin{equation}
	G = \bigl( (-b,-a) \cup (a,b) \bigr) \times \bigl( (-b,-a) \cup (a,b) \bigr).
\end{equation}
For any ${\eps > 0}$ we can find $L^1$-$\eps$-flat sums in $G$ with a frequency function
$\om$ given by equation~\eqref{eToDimFrequencyF}. 
\end{lem}


Following the terminology introduced by A.\,Danilenko (see~\cite{DanilenkoOnFunnyRankOneWMix}) 
we define our rank one action of the group $\Set{R}^2$ to be the rank one action 
given by a $(C,F)$-construction, where at $n$-th step of the construction: 
$C_n$ is an open set in $\Set{R}^2$ (analogue of the tower) and $F_n$ is a finite set. 
We set precisely 
\begin{gather}
	C_n = (0,h_n) \times (0,h_n), 
	\\ 
	F_n = \bigl\{ \bigl(\om_X(y_1,y_2),\om_Y(y_1,y_2)\bigr) \where (y_1,y_2) \in \{0,1,\dots,q_n-1\}^{\times 2} \bigr\}. 
\end{gather}

\begin{lem}
Given two functions $f_1$ and $f_2$ with zero mean the spectral measure 
${\sigma_{f_1 \otimes f_2}}$ with respect to the action $\T \otimes \T$ in the space ${H \otimes H}$
is absolutely continuous with respect to the Lebesgue measure. 
And the measure of maximal spectral type for $\T \otimes \T$ 
in the space $H^{(2)}$ of all functions with zero mean is~equivalent to the sum of the Lebesgue measure 
and a singular measure with support on two coordinate axes, 
\begin{equation}
	\sigma \sim \la + \df_0(\tau_1) + \df_0(\tau_2). 
\end{equation}
\end{lem}

The serious difference between dimension one and dimension ${d > 2}$ is that 
the trivial subspace of constants in $L^2(X)$ produces a non-trivial 
subspace ${\{\const\} \times H}$ in~$H^{(2)}$. 

%

\subsection{$\Set{R}^d$-actions with simple Lebesgue spectrum} 

\begin{figure}[th]
	\centering
  \unitlength=1mm
  \includegraphics[height=9cm]{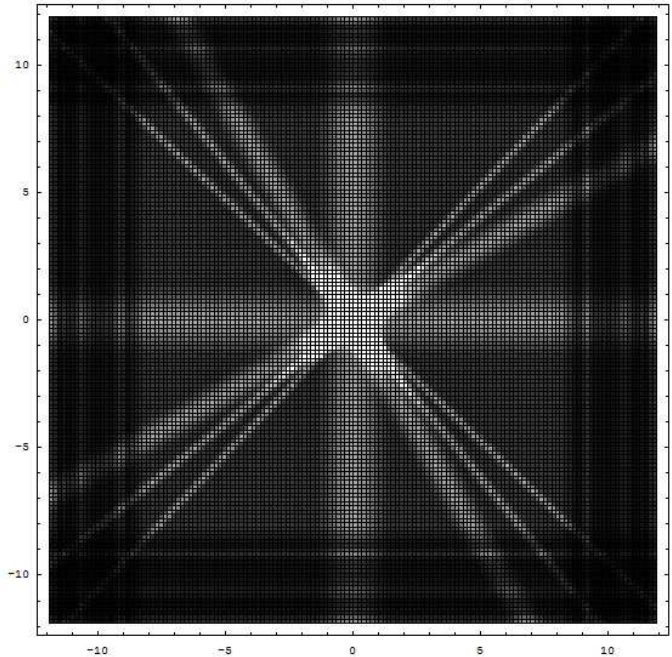} 
  \caption{The structure of the Riesz product (density plot of $\prod_n \cP_n(\tau)$ graph)} 
  \label{fStrOfRieszProdDPlot}
\end{figure}

In the above discussion we have described the obstacle to simple Lebesgue spectrum for 
the frequency function ${\omega(y_1,y_2) = q\beta^{-2}(e^{\beta y_1/q}, e^{\beta y_2/q})}$. 
Two~components $\df(x_1)$ and $\df(x_2)$ appear in addition to Lebesgue component. 


The idea of the next construction is to overcome this obstacle by choosing a~sequence 
of {\em different\/} coordinate systems given by some linear transforms $\Psi_n$ 
in such a~way that at any step of the construction the polynomials $\cP_n(\tau)$ is rotated 
on the plane $\Set{R}^2$ according the transform $\Psi_n$, 
\begin{equation}
	\Phi(\tau) = \prod_n \cP_n(\Psi_n \tau) 
\end{equation}
and the ``bad set'' 
\begin{equation}
	G_n^c = \Set{R}^2 \sms G_n = 
		(\Set{R}^2 \sms [-b_n,b_n]^{\times 2}) \cup 
		([-b_n,b_n] \times [-a_n,a_n]) \cup 
		([-a_n,a_n] \times [-b_n,b_n]) 
\end{equation}
for the polynomial $\cP_n(\tau)$ after applying $\Psi_n^{-1}$ 
is covered (in~most part) by the windows 
\begin{equation}
	\Psi_{n+1}^{-1} G_{n+1}, \quad \Psi_{n+2}^{-1} G_{n+2}, \quad \ldots 
\end{equation}
In other words, the part of the Riesz product $\prod_{k\:<\:n} |\cP_k(\tau)|^2$ 
which is {\em not necessary flat\/} is covered by the area of flatness 
of the polynomials $\cP_{n+k}(\Psi_{n+k}\tau)$, ${k \ge 1}$, 
so that it becomes in a sense ``frozen'' for all further steps of the construction. 

Remark that $G_n^c$ consists of the area outside a~big square 
$[-b_n,b_n]^{\times 2}$ and two thin rectangles 
\begin{equation}
	\cF_n = 
		([-b_n,b_n] \times [-a_n,a_n]) \cup 
		([-a_n,a_n] \times [-b_n,b_n]) 
\end{equation}
located near coordinate axes 
(recall that ${b_n \to \infty}$ and ${a_n \to 0}$ as ${n \to \infty}$).
We will find appropriate $\Psi_n$ to achieve the following effect: 
\begin{equation}
	\max\{ |\tau| \where \tau \in \Psi_n^{-1}\cF_{n} \cap \Psi_{n+k}^{-1}\cF_{n+k} \} \to 0 
	\quad \text{as $n \to \infty$}, \qquad k \ge 1
\end{equation}

\begin{figure}[th]
  \centering
  \unitlength=1mm
  \includegraphics{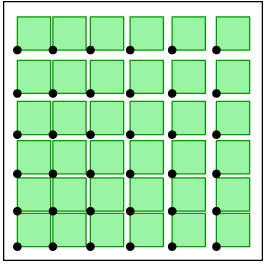} 
  \caption{Illustration for the set $F_1$. 
  	The small squares are shifts of the set $C_1$, 
  	and the large bounding square is the set $C_2$. 
  	One can see that the distance between adjacent points slightly increases from the left to the right
  	and from the bottom to the top.} 
  \label{fRankOneFlowFnA}
\end{figure}

\begin{figure}[th]
  \centering
  \unitlength=1mm
  \includegraphics{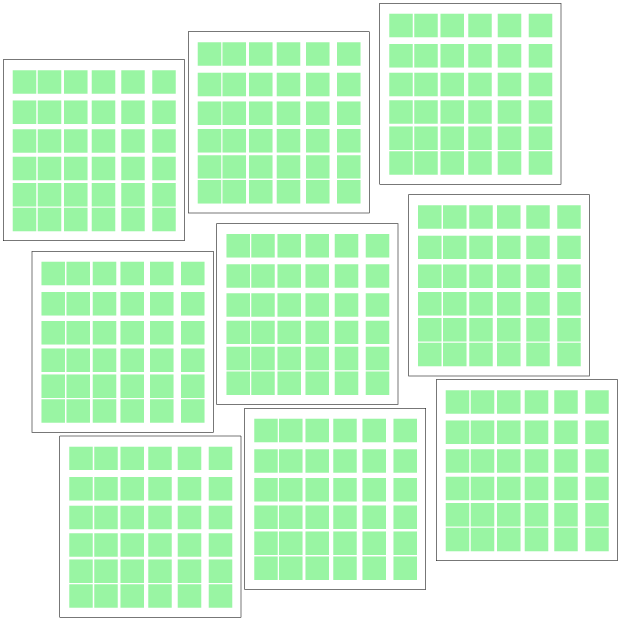} 
  \caption{Illustration to construction~\ref{cRActionLS}, the second step.} 
  \label{fRankOneFlowFnA}
\end{figure}

\def\ee{\Hat e}

\begin{constr}\label{cRActionLS}
Assume that for each index $n$ a pair of basis vectors $\{\ee_1(n),\ee_2(n)\}$ is given 
such that 
\begin{gather}
	|\ee_1(n)| = |\ee_2(n)| = 1, \\
	\ee_1(n) = \Psi_n \begin{pmatrix} 1 \\ 0 \end{pmatrix}, \qquad
	\ee_1(n) = \Psi_n \begin{pmatrix} 0 \\ 1 \end{pmatrix}, 
\end{gather}
and let $\Psi_{n+1} = V_n\Psi_n$ be a small linear correction of $\Psi_n$ 
by a~map~$V_n$ to be defined later. 
%
%
Set 
\begin{equation}
	F_n = \bigl\{ 
		\om_{X,n}(y_1)\,\ee_1(n) + \om_{Y,n}(y_2)\,\ee_2(n) \where y \in \{0,\dots,q_n-1\}^{\times 2} 
	\bigr\}
\end{equation}
Let us define $\ee_1(n)$ and $\ee_2(n)$ as follows. Imagine that everuthing is consider 
in the coordinate system connected to the original one via the transform $\Psi_n$. 
In this coordinate system $\ee_j(n)$ becomes $(1,0)$ and $(0,1)$ 
and $\Psi_n^{-1} F_n$ is exactly the non-perturbed set $F_n$ 
comming from the beginning example of $\T \otimes \T$. 
It~looks like a~rectangular grid ${F_n^{(1)} \times F_n^{(1)}}$ 
so that the adjacent points in this set connected by a~vector ${v \approx (h_n,0)}$ 
or ${v \approx (0,h_n)}$. 
Now let us set
\begin{gather}
	h_{n+1}\, \ee_1(n+1) = \ell_n \ee_1(n) + \xi_n \ell_n e_2(n), \quad  
	h_{n+1}\, \ee_2(n+1) = -\xi_n \ell_n e_2(n) + \ell_n e_1(n), 
	\\ 
	\ell_n = q_n h_n e^{\beta_n}. 
\end{gather}
(Notice that for the original non-perturbed construction of $\T \otimes \T$ we would have ${\xi_n = 0}$.) 

Thus, one can get the following representation for the elementary rotation on each step:
\begin{equation}
	V_n = 
	\begin{pmatrix}
		1 & -\xi_n \\
		\xi_n & 1
	\end{pmatrix}. 
\end{equation}
Such kind of cunstruction in the context of $\Set{Z}$-actions was used by V.\,Ryzhikov in~\cite{RyzhikovOnRHLWEps}. 
The spectral measure $\sigma_f$ can be represented in a form of a Riesz product, 
which is formal so far, and our purpose is to prove its convergence: 
\begin{equation}
	\sigma_n = |f_{(n_0)}|^2 \, \prod_{n=n_0}^\infty |\cP_n(\tau)|^2 \,d\tau, \qquad 
	\cP_n = \frac1{q_n}\,\Hat{\bf 1}_{F_n}, 
\end{equation}
where 
\begin{equation}
	{\bf 1}_{F_n} 
		= \sum_{t \:\in\: F_n} \df_t 
		= \sum_{y_1=0}^{q_n-1} \sum_{y_2=0}^{q_n-1} \df_{\om(y_1,y_2)\Psi^{-1}_n}, 
\end{equation}
According to this modification applied to $F_n$ let us define $C_n$ to be the open set 
$\Psi_n((0,h_n) \times (0,h_n))$. 
\end{constr}

\begin{thm}\label{thmSimpleLebSpRd}
The rank one $\Set{R}^2$-action built in construction~\ref{cRActionLS} has simple Lebesgue spectrum 
if $q_n$ go~to~infinity sufficiently fast. 
\end{thm}

\begin{proof}
It is enough to establish that $\sigma_f$ is absolutely continuous for any $f$ constant 
on the levels of the $n_0$'th tower. If we compare this action with the construction 
of rank one flow described above all the points on the plane can be classified to the 
following three groups: 
\begin{itemize}
	\item[(a)] points covered by many (more than one) sets $\cF_n$, 
		where flatness of~$\cP_n(\tau)$ cannot be controlled;
	\item[(b)] points containing infinitely many $\cF_n$ entering its arbitrary small boundary;
	\item[(c)] points with a boundary free from points from $\cF_n$ starting from some $n^*$; 
\end{itemize}
Notice that the case (b) is not present for the flow, but in the case of the $\Set{R}^2$-action 
the case (b) is observed on the limit set of $\cF_n$ which is non-empty and 
consists of two {\it limit lines\/} crossing at zero (see~figure~\ref{fStrOfRieszProdDPlot}). 
It can be easily seen that the intersection of $\cF_n$ and $\cF_m$ can be fit in a ball 
of radius 
\begin{equation}
	r_{n,m} = O\!\left( \frac{a_n + a_m}{\varphi_{n,m}} \right), 
\end{equation}
where 
$a_n$ is the thickness of the rectangular strips in~$\cF_n$, and $\varphi_{n,m}$ is 
the angle between them, 
\begin{equation}
	\varphi_{n,m} \sim \xi_n + \xi_{n+1} +\ldots+ \xi_{m-1} \le \sum_{k \ge n} \xi_k 
\end{equation}
if ${n < m}$. Since $\xi_n$ are fixed in advance, before we construct the sequence 
of flat polynomials, we have to require the following condition:
\begin{equation}
	a_n \ll \sum_{k \ge n} \xi_k 
\end{equation}
to ensure that all these intersections collapses to the zero point, so that we can apply 
the same arguments like in the case of rank one flow and see that atom at zero 
in the limit distribution prohibited by the ergodicity of the actions. 

At the same time, case (b) appear only in the multidimensional case 
and to complete the proof we have to use the following lemma on 
the polynomials used in the construction 
(see~lemma~\ref{thmLOneCompactFlatness}). 

\begin{lem}\label{lemLOneEstNearZero}
A flat polynomial $\cP(\tau_1)$ on $\Set{R}$ build for the function 
\begin{equation}
	\om(y) = M\frac{q}{\beta^2} e^{\beta y/q}. 
\end{equation}
can be estimated near zero as follows: 
\begin{equation}
	\|\cP\|_{L^1(-a,a)} = O\!\left( \frac{\log aq}{\sqrt{q}} \right), \qquad a < 1.  
\end{equation}
\end{lem}

Using this lemma one can see that for any small neighbourhood $U_0$ 
of a~limit line the intersections $U_0 \cap \cF_n$ become disjoint starting at some indedx~$n_1$ 
and 
\begin{equation}
	\sum_{n \ge n_1} \|\cP_n|_{U_0 \cap \cF_n}\|_1 < \infty. 
\end{equation}
This calculation cannot be applied to the case (a), since the polynomials $\cP_n(\tau)$ 
are multiplied and if we loose disjointness of $\cF_n$ it will be impossible to control 
the Riesz product behaviour. 

The case (c) corresponds to the area of flatness. 
\end{proof}

We have considered only the case of rank one $\Set{R}^2$-actions 
and the same effects remain for $\Set{R}^d$-actions. 
Thus, the existence of rank one actions with Lebesgue spectrum 
is established using the same technique. 

\bigskip
The author is very greatly to El.~Houcein El~Abdalaoui, Bassam Fayad, Mariusz Lemanczyk 
and Jean-Paul Thouvenot for the deep and helpful discussions concerning the first part of this paper.



\bibliographystyle{amsplain}
\bibliography{IcePaperI}
 
\end{document}